\documentclass{amsart}
\usepackage{latexsym,amsfonts}

\newtheorem{theorem}{Theorem}[section]
\newtheorem{lemma}[theorem]{Lemma}
\def\co{\colon\thinspace}
\newcommand{\cser}{\mathcal{C}}
\DeclareMathOperator{\Leib}{Leib}

\begin{document}

\title{On Engel's Theorem for Leibniz algebras} 
\author{Donald W. Barnes}
\address{1 Little Wonga Rd, Cremorne NSW 2090 Australia}
\email{donwb@iprimus.com.au}


\begin{abstract}  I give a simpler proof of the generalisation of Engel's Theorem to Leibniz algebras. 
\end{abstract}

\maketitle

\section{Introduction}

Let $L$ be a finite-dimensional left Leibniz algebra.  I denote by $d_a$ the left multiplication operator $d_a\co L \to L$ defined by $d_a(x) = ax$ for all $a, x \in L$.
Ayupov and Omirov \cite[Theorem 2]{AyO} prove: 
\begin{theorem}  Suppose that the left multiplication operator
$d_a$ is nilpotent for all $a \in L$.  Then $L$ is nilpotent. \end{theorem}

Let $V$ be an $L$-module.  I denote the left action of $L$ on $V$ by $T$ and the right by $S$, thus $T_a(v) = av$ and $S_a(v) = va$.
Patsourakos obtains the above theorem as a corollary of the stronger result \cite[Theorem 7]{Pats}:

\begin{theorem} \label{pats} Let $(S,T)$ be a representation of $L$ on a vector space $V\ne 0$ such that $T_a$ is nilpotent for all $a \in L$.  Then $S_a$ is nilpotent for all $a \in L$ and there exists $v \in V$, $v \ne 0$ such that $T_a(v) = S_a(v) = 0$ for all $a \in L$.\end{theorem} 

I call the subspace $\langle x^2 \mid x \in L \rangle$ spanned by the elements of $L$ the Leibniz kernel of $L$ and denote it $\Leib(L)$.

\begin{lemma} \label{th-leib} The subspace $\Leib(L) = \langle x^2 \mid x \in L \rangle$  is an ideal of $L$.
\end{lemma}

\begin{proof}  For all $x,y \in L$, we have $x^2y = 0$ and $xy^2 = (x+y^2)^2 - x^2$.
\end{proof}

I give a simpler proof of Patsourakos's Theorem.  For this, I need a result which is implicit in Loday and Pirashvili (1996).  I denote by $\cser_L(V)$ the kernel of the representation of $L$ on $V$, that is, the centraliser in $L$ of $V$ in the split extension of $V$ by $L$.

\begin{theorem}\label{irred}  Let $L$ be a finite-dimensional Leibniz algebra and let $V$ be a finite-dimensional irreducible $L$-bimodule.  Then $L/\cser_L(V)$ is a Lie algebra and either $VL=0$ or $vx = -xv$ for all $x \in L$ and $v \in V$. \end{theorem}

\begin{proof}  Without loss of generality, we may suppose $\cser_L(V) = 0$.  Form the split extension $X$ of $V$ by $L$ and let $K = \Leib(X)$.  Since $V$ is a minimal ideal of $X$, either $K \ge V$ or $K \cap V = 0$.  As $K$ is an abelian ideal, in either case, we have $K \le \cser_X(V) = V$.  For all $x \in L$, $x^2 \in K$, so $x^2 = 0$.  Thus $L$ is  a Lie algebra.  Either $K= V$ or $K = 0$.  If $K = V$, then $VX = 0$.  If $K = 0$, then $X$ is a Lie algebra, so $vx = -xv$ for all $x \in L$ and $v \in V$.
\end{proof}
   
\section{Proof of Theorem \ref{pats}}
\begin{proof} If $W$ is a submodule of $V$, it satisfies the conditions of Theorem \ref{pats}, so we may suppose that $V$ is irreducible.  By working modulo the kernel of the representation, we may suppose that $V$ is faithful.  By Theorem \ref{irred}, $L$ is a Lie algebra and $V$ is either symmetric ($vx = -xv$ for all $x\in L $ and $v \in V$), or antisymmetric ($vx=0$ for all $x \in L$ and $v \in V$).  By Engel's Theorem for Lie algebras of linear transformations, there exists $v \in V$, $v \ne 0$ such that $xv = 0$ for all $x \in L$.  In either case, we have $vx = 0$ for all $x \in L$.
\end{proof}

\bibliographystyle{amsplain}

\end{document}